\documentclass[10pt,oneside]{amsart}

\usepackage{latexsym,amssymb,amsmath,amscd,amsxtra,amsfonts,amsthm}
\usepackage{fancyhdr,array,dsfont}

\newtheorem{theorem}{Theorem}[section]
\newtheorem{lemma}[theorem]{Lemma}
\newtheorem{proposition}[theorem]{Proposition}

\newtheorem{remark}[theorem]{Remark}

\newcommand{\beq}{\begin{eqnarray*}}
\newcommand{\eeq}{\end{eqnarray*}}
\newcommand{\beqn}{\begin{eqnarray}}
\newcommand{\eeqn}{\end{eqnarray}}

\numberwithin{equation}{section}

\begin{document}

%%%%% To ease editing, for IMPAN journals add:

%\baselineskip=17pt

%%%%%%%%%%%

%% In the running head, replace first names by initials
%% and give an abbreviation of the title.

\title[Poincar\'e and quadratic transportation-variance inequalities]{The Poincar\'e inequality and quadratic transportation-variance inequalities}

\author[Y. LIU]{Yuan LIU}
\address{$^{1}$Institute of Applied Mathematics, Academy of Mathematics and Systems Science,
Chinese Academy of Sciences, Beijing 100190, China}
\address{$^{2}$University of Chinese Academy of Sciences, Beijing 100049, China}
\email{liuyuan@amss.ac.cn}

\date{\today}

\begin{abstract}
It is known that the Poincar\'e inequality is equivalent to the quadratic transportation-variance inequality (namely $W_2^2(f\mu,\mu) \leqslant C_V \mathrm{Var}_\mu(f)$), see Jourdain \cite{Jourdain} and most recently Ledoux \cite{Ledoux18}. We give two alternative proofs to this fact. In particular, we achieve a smaller $C_V$ than before, which equals the double of Poincar\'e constant. Applying the same argument leads to more characterizations of the Poincar\'e inequality. Our method also yields a by-product as the equivalence between the logarithmic Sobolev inequality and strict contraction of heat flow in Wasserstein space provided that the Bakry-\'Emery curvature has a lower bound (here the control constants may depend on the curvature bound).

Next, we present a comparison inequality between $W_2^2(f\mu,\mu)$ and its centralization $W_2^2(f_c\mu,\mu)$ for $f_c = \frac{|\sqrt{f} - \mu(\sqrt{f})|^2}{\mathrm{Var}_\mu (\sqrt{f})}$, which may be viewed as some special counterpart of the Rothaus' lemma for relative entropy. Then it yields some new bound of $W_2^2(f\mu,\mu)$ associated to the variance of $\sqrt{f}$ rather than $f$. As a by-product, we have another proof to derive the quadratic transportation-information inequality from Lyapunov condition, avoiding the Bobkov-G\"otze's characterization of the Talagrand's inequality.
\end{abstract}



\maketitle

\allowdisplaybreaks

\section{Introduction}
 \label{Intro}
 
The aim of this paper is to investigate some links between the Poincar\'e inequality (PI for short) and various comparison inequalities of quadratic Wasserstein distance with variance. Some conclusions might be extended to abstract settings of metric measure spaces, nevertheless for simplicity, our basic framework is specified as follows. Let $E$ be a connected complete Riemannian manifold of finite dimension, $d$ the geodesic distance, $\mathrm{d}x$ the volume measure, $\mathcal{P}(E)$ the collection of all probability measures on $E$, $\mu(\mathrm{d}x) = e^{-V(x)}\mathrm{d}x \in \mathcal{P}(E)$ with $V\in C^1(E)$, $\mathrm{L}=\Delta - \nabla V\cdot \nabla$ the $\mu$-symmetric diffusion operator with domain $\mathbb{D}(\mathrm{L})$, and $\Gamma(f,g) = \nabla f\cdot \nabla g$ the carr\'{e} du champ operator with domain $\mathbb{D}(\Gamma)$, satisfying the integration by parts formula
   \[ \int \Gamma(f,g) \;\mathrm{d}\mu = -\int f \mathrm{L}g  \;\mathrm{d}\mu, \ \forall f\in \mathbb{D}(\Gamma), g\in \mathbb{D}(\mathrm{L}). \]
Define the $L^p$ Wasserstein (transportation) distance (also called Kantorovich metric) between $\nu, \mu\in \mathcal{P}(E)$ for any $p\geqslant 1$ by
    \[ W_p(\nu,\mu) = \left(\inf\limits_{\pi\in \mathcal{C}(\nu,\mu)} \int_{E\times E} d^p(x,y) \pi(\mathrm{d}x, \mathrm{d}y)\right)^{1/p}, \]
where $\mathcal{C}(\nu,\mu)$ denotes the set of any coupling $\pi$ on $E\times E$ with marginals $\nu$ and $\mu$ respectively. Throughout this paper we focus on quadratic Wasserstein distance, so it is convenient to assume $\mu$ has a finite moment of order $2$. The reader is referred to several constant references as Bakry-Gentil-Ledoux \cite{BGL} and Villani \cite{Villani1, Villani2} for detailed presentations.

Our motivation partially arises from the problem of how to characterize the exponential decay of quadratic Wasserstein distance along heat flow. It is known that the exponential decay of heat semigroup $P_t = \exp(t\mathrm{L})$ in $L^2$-norm is equivalent to PI, which reads for any $f\in \mathbb{D}(\Gamma) \cap L^2(\mu)$
   \[  \mathrm{Var}_\mu(P_tf)  \leqslant e^{-2t/C_P}\mathrm{Var}_\mu(f) \ \Longleftrightarrow\  \mathrm{Var}_\mu(f)  \leqslant C_P\int \Gamma(f,f)  \mathrm{d}\mu \]
(simply denote by $\mu(h) = \int h \mathrm{d}\mu$ the expectation and by $\mathrm{Var}_\mu(f) = \mu(f^2) - (\mu(f))^2$ the variance). Similarly, the exponential decay of $P_t$ in relative entropy is equivalent to the logarithmic Sobolev inequality (LSI for short), which reads for any $f>0$ with $\sqrt{f}\in \mathbb{D}(\Gamma)$
    \[ \mathrm{Ent}_\mu(P_tf) \leqslant e^{-2t/C_{LS}}\mathrm{Ent}_\mu(f) \ \Longleftrightarrow\ \mathrm{Ent}_\mu(f)  \leqslant \frac12C_{LS} \mathrm{I}_\mu(f) \]
(denote by $\mathrm{Ent}_\mu(f) = \int f\log f  \mathrm{d}\mu$ the relative entropy and by $\mathrm{I}_\mu(f) = \int  \frac{\Gamma(f,f)}{f} \mathrm{d}\mu$ the Fisher information). Somehow, we think it is tough to give a proper answer to the same question in Wasserstein space, namely to find some equivalent inequality characterizing $W_2^2(P_t\nu, \mu) \leqslant e^{-2\kappa t} W_2^2(\nu, \mu)$ (or up to a multiple)  with $\kappa >0$ for any $\nu = f\mu  \in \mathcal{P}(E)$. When we turn to some weak replacements, one natural candidate is to compare $W_2$ with variance,  which can be quickly derived from the control inequality of weighted total variation (see \cite[Proposition 7.10]{Villani1}) and H\"older inequality that 
   \[ W_2^2(\nu, \mu) \leqslant 2||d^2(x_0,\cdot)(\nu-\mu)||_{\textrm{TV}} \leqslant 2\int d^2(x_0, \cdot)\left| f-1 \right|  \mathrm{d}\mu
         \leqslant C \sqrt{\mathrm{Var}_\mu(f)} \]
if $d^4(x_0,\cdot)$ is $\mu$-integrable. At least, it follows the integrability of $W_2^2(P_t\nu, \mu)$ for $t\in[0,\infty)$ provided that PI holds true, which is helpful to the semigroup analysis more or less.

If $\mu$ fulfills the Talagrand's inequality ($\textrm{W}_2\textrm{H}$ for short), namely the control of relative entropy on $W_2(\nu, \mu)$ as
    \[ W_2^2(\nu, \mu) \leqslant 2C_T\mathrm{Ent}_\mu(f), \]
it follows from the preliminary inequality $\mathrm{Ent}_\mu(f) \leqslant p\left( \mathrm{Var}_\mu(f) \right)^{\frac1p}$ for $p\geqslant 1$ that
    \[ W_2^2(\nu, \mu) \leqslant 2C_T p\left( \mathrm{Var}_\mu(f) \right)^{\frac1p}. \]
In particular, for $p=2$ it covers $W_2^2(\nu, \mu) \leqslant C \sqrt{\mathrm{Var}_\mu(f)}$, and for $p=1$ it gives
   \begin{eqnarray}
      W_2^2(\nu, \mu) \leqslant 2C_T \mathrm{Var}_\mu(f), \label{eq0-1}
   \end{eqnarray}
which suggests an improved decay rate of $W_2$ along heat flow. Since $\textrm{W}_2\textrm{H}$ implies PI with $C_P\leqslant C_T$ (see \cite{BGL} for example), it is natural to ask what about the relation between PI and a transportation-variance inequality like (\ref{eq0-1}). Indeed, Jourdain \cite{Jourdain} proved their equivalence  in dimension one. Ding \cite{Ding} claimed a general inequality between $W_2$ and the so called R\'{e}nyi-Tsallis divergence of order $\alpha$, which equals the variance for $\alpha=2$ (somehow, it is obscure for us to check Remark 3.3 therein for small variance, maybe we misunderstand something). Then Ledoux \cite{Ledoux18} provided a very streamlined proof to show a general result that PI is equivalent to the quadratic transportation-variance inequality ($\textrm{W}_2\textrm{V}$ for short)
    \[ W_2^2(\nu, \mu) \leqslant C_V \mathrm{Var}_\mu(f) \]
for $C_V\leqslant 4C_P$. We give two alternative proofs to this fact and achieve a smaller constant as $C_V\leqslant  2C_P$. Conversely, various  perturbation techniques ensure PI with a constant no more than $C_V$ if assume $\textrm{W}_2\textrm{V}$ (see \cite{Ledoux18}). Precisely, our first main result is the following.

\begin{theorem} \label{Thm1}
Let $\nu=f\mu  \in \mathcal{P}(E)$. The Poincar\'e inequality implies next every inequality:
\begin{enumerate}
\item $W_2^2(\nu, \mu) \leqslant 2C_P \sqrt{\mathrm{Var}_\mu(f)}\cdot\sqrt{\mathrm{Ent}_\mu(f)}$.

\item $W_2^2(\nu, \mu) \leqslant 2C_P \mathrm{Var}_\mu(f)$.

\item $W_2^2(\nu, \mu) \leqslant 2C_P \inf\limits_{p\geqslant 1} \left\{p^2\left(\mathrm{Var}_\mu(f)\right)^{\frac1p}\right\}$.

\item $W_2^2(\nu, \mu) \leqslant 2C_P \inf\limits_{p\geqslant 1} \left\{p^2 {\big(} C_P\mu(\Gamma(f,f)) {\big)}^{\frac1p}\right\}$.

\item $W_2^2(\nu, \mu) \leqslant 2C_P^2\mu(\Gamma(f,f))$.
\end{enumerate}
Conversely, the above every one implies the Poincar\'e inequality with constant $\sqrt{2}C_P$.
\end{theorem}
\begin{remark}
If assume (1) or (5) prior to PI, the perturbation technique ensures PI with constant $\sqrt{2}C_P$. Note that the same technique doesn't work for (2) directly. 
\end{remark}

There are two approaches to this end, and both are contributed to get the inequality (see also (\ref{eq2-basic}) below)
   \[ W_2^2(\nu, \mu) \leqslant 2\sqrt{\mathrm{Ent}_\mu(f)} \int_0^{\infty} \sqrt{\mathrm{Ent}_\mu(P_tf)}\mathrm{d}t. \]
The first approach is a shortcut based on the interpolation technique developed by Kuwada \cite{Kuwada} and further by \cite{Ledoux18}. The other one appeals to the derivative formula of $W_2^2(P_tf\mu,\mu)$ in $t$ (almost everywhere), which is slightly different from what Otto-Villani employed in \cite[Lemma 2]{Otto-Villani}. Our method doesn't involve the theory of solving Fokker-Planck equation on Riemannian manifolds, so we have a by-product as reproving their lemma for nice initial data but avoiding the curvature condition. 

Another by-product is to show the equivalence between the LSI and strict contraction of heat flow in Wasserstein space (here we actually mean a strictly exponential decay of $W_2(P_tf\mu,\mu)$ with some multiple in front) provided that the Bakry-\'Emery curvature has a lower bound. One can compare the following with the well known characterization of curvature-dimension condition through the heat flow contraction (see \cite[Theorem 9.7.2]{BGL} for this fact and \cite[Subsection 3.4.5]{BGL} for precise definition of curvature-dimension condition $CD(\rho,\infty)$).

\begin{proposition} \label{Prop1}
Assume $V$ is a smooth potential such that the curvature-dimension condition $CD(\rho,\infty)$ holds for $\rho\in\mathbb{R}$. Then the next two statements are equivalent:
\begin{enumerate}
\item there exist two constants $C>0$ and $\kappa>0$ such that for all $t>0$ and any $\nu = f\mu \in \mathcal{P}(E)$ 
              \[ W_2(P_t\nu,\mu) \leqslant  Ce^{-\kappa t} W_2(\nu,\mu); \]
\item there exists a constant $C_{LS} >0$ such that the LSI holds.
\end{enumerate}
\end{proposition}
\begin{remark}
The constants involved here may depend on $\rho$. If the LSI holds, we have $\kappa = 1/C_{LS}$. Very recently, Wang \cite{Wang2} discussed exponential contraction in any $W_p \ (p\geqslant 1)$ for a class of diffusion semigroups and gave the implication from (2) to (1) as well.
\end{remark}

Next, we are interested in the comparison of $W_2^2(\nu,\mu)$ to $\mathrm{Var}_\mu(\sqrt{f})$ rather than $\mathrm{Var}_\mu(f)$. In general, one can't expect a strong inequality as $W_2^2(\nu, \mu) \leqslant C\mathrm{Var}_\mu(\sqrt{f})$, since from PI it follows $W_2^2(\nu, \mu)\leqslant \frac14CC_P \mathrm{I}_\mu(f)$, which is called the quadratic transportation-information inequality ($\textrm{W}_2\textrm{I}$ for short, see \cite{GLWY}), and it is known that $\textrm{W}_2\textrm{I}$ is strictly stronger than PI and even than $\textrm{W}_2\textrm{H}$. Actually what we present first is a new inequality between the Wasserstein distance and its ``centralization", which may be viewed as a special counterpart of the Rothaus' lemma for relative entropy (see \cite[Lemma 5.1.4]{BGL}), namely for any $a\in \mathbb{R}$
   \[ \mathrm{Ent}_\mu\left((h+a)^2\right) \leqslant \mathrm{Ent}_\mu(h^2) + 2\mu(h^2). \]
Precisely we have
\begin{theorem} \label{Thm2}
Let $\nu=f\mu$, $c=\mu(\sqrt{f})$ and $\sigma^2 = \mathrm{Var}_\mu(\sqrt{f})$. Let $f_c = \frac{|\sqrt{f} - c|^2}{\sigma^2}$. If the Poincar\'e inequality holds, then there exists two constants $C_1$ and $C_2$ such that
   \[ W_2^2(\nu, \mu) \leqslant C_1 \sigma^2W_2^2(f_c\mu,\mu) +C_2 \sigma^2. \]
\end{theorem}

\begin{remark}
For instance, we can take $C_1 = 2$ and $C_2 = 96C_P$. Actually our method implies that $C_1$ can approach $1$ but should be strictly greater than $1$. Moreover, $f_c$ can be extended to $f_\theta = \frac{|\sqrt{f} - \theta|^2}{\mu((\sqrt{f}-\theta)^2)}$ for any $\theta\in(0,2c)$ associated with two constants $C_1(\theta)$ and $C_2(\theta)$ depending on $\theta$.
\end{remark}

As consequence, when $E$ has a finite diameter, it follows by the definition of $W_2$
    \begin{eqnarray}
        W_2^2(\nu, \mu) \leqslant  \sigma^2 \left(C_1(\textrm{diam}E)^2 + C_2\right),  \label{eqBdd0}
    \end{eqnarray}
which can't be directly concluded by Theorem \ref{Thm1} we think. Then it quickly derives $\textrm{W}_2\textrm{I}$ from PI again. Moreover, a LSI holds by using the HWI inequality in \cite{Otto-Villani, Villani1, BGL} under the curvature-dimension condition $CD(\rho,\infty)$, with the control constant $C_{LS} = \lambda{\big(}(1- \frac{\rho}{4}\lambda) \vee 1{\big)}$ for $\lambda=\sqrt{C_P(C_1(\textrm{diam}E)^2 + C_2)}$. There is a lot of literature concerning LSI, for example one can compare the above (\ref{eqBdd0}) with \cite[Theorem 1.4]{Wang} about the constant estimate on compact manifolds by means of semigroup analysis. 

When $E$ is unbounded, we have at least by using \cite[Proposition 7.10]{Villani1} that
    \begin{eqnarray}
       W_2^2(\nu, \mu) \leqslant  C \left(\sigma^2 + \int d^2(x_0,\cdot) (\sqrt{f} - c)^2 \mathrm{d}\mu \right).  \label{eqUBdd1}
    \end{eqnarray}
It gives a direct way to derive $\textrm{W}_2\textrm{I}$ from the so-called Lyapunov condition. Recall \cite{Liuy}, the Lyapunov condition here means there exists such a function $W>0$ satisfying that $W^{-1}$ is locally bounded and for some $c>0, b\geqslant 0$ and $x_0\in E$ holds in the sense of distribution
  \begin{eqnarray}
    \mathrm{L}W \leqslant \left(-cd^2(x,x_0) + b\right)W. \label{eqLya}
  \end{eqnarray}
Partial proof in \cite{Liuy} applied the Bobkov-G\"otze's characterization of $\textrm{W}_2\textrm{H}$, namely there is a constant $C>0$ such that $\mu\left(\exp(Q_Ch)\right) \leqslant \exp\left(\mu(h)\right)$ for all $h\in L^\infty(\mu)$, where $Q_C$ denotes the infimum-convolution operator and $Q_Ch$ solves the Hamilton-Jacobi equation $\frac{\mathrm{d}}{\mathrm{d}t} Q_th + \frac12 |\nabla Q_th|^2 = 0$ for initial data $h$, see \cite{BGL, BG} for example. Nevertheless, facing the stability problem for $\textrm{W}_2\textrm{H}$ under bounded perturbation, one needs various additional curvature conditions so far, for example see \cite{GRS, EM}. When we turn to the same problem for $\textrm{W}_2\textrm{I}$, it would be more robust if we can find a direct method to derive $\textrm{W}_2\textrm{I}$ from (\ref{eqLya}) with no appearance of $\textrm{W}_2\textrm{H}$. Actually, Theorem \ref{Thm2} takes on such a role.

The paper is organized as follows. In next Section \ref{FPT}, we give a quick proof to Theorem \ref{Thm1}. In Section \ref{DHF} and  \ref{PT1}, we compute the derivative of quadratic Wasserstein distance along heat flow, and then complete the other proof of Theorem \ref{Thm1}. The equivalence of the LSI and strict contraction of heat flow in Wasserstein space is shown in Section \ref{ELSI}. Section \ref{CWD} is devoted to the comparison inequality about centralization of quadratic Wasserstein distance, and Section \ref{ATI} provides a direct proof of  $\textrm{W}_2\textrm{I}$ under the Lyapunov condition.

\bigskip

\section{The first proof of Theorem \ref{Thm1}}
 \label{FPT}

Recall that, for any bounded Lipschitz function $h$, define its infimum-convolution for any $t>0$ by
    \[ Q_t h(x) := \inf\limits_{y} \left\{h(y) + \frac{1}{2t} d^2(x,y) \right\}, \]
 which solves the Hamilton-Jacobi equation (see for example \cite[Section 9.4]{BGL}, \cite[Section 3.3]{Evans}, \cite[Section 5.4]{Villani1})
    \[  \left\{\begin{array}{l}
               \frac{\mathrm{d}}{\mathrm{d}t} u + \frac12|\nabla u|^2 = 0,\\
            
               u(x,0) = h(x).
        \end{array}\right. \]
According to \cite{Kuwada, Ledoux18}, for any decreasing function $\lambda \in C^1[0,+\infty)$ with $\lambda(0) = 1$ and $\lim\limits_{t\to \infty} \lambda(t) = 0$, one has a semigroup interpolation by virtue of Hamilton-Jacobi equation, integration by parts and the H\"older inequality that
   \begin{eqnarray*}
      \int_E Q_1h f\mathrm{d}\mu - \int_E h \mathrm{d}\mu &=& \int_E\int_0^{\infty} -\frac{\mathrm{d}}{\mathrm{d} t} Q_{\lambda} h P_t f \mathrm{d}t \mathrm{d}\mu\\
      &=&  \int_E\int_0^{\infty} \frac12\lambda' |\nabla Q_{\lambda} h|^2 P_t f  -  Q_{\lambda} h \cdot \mathrm{L} P_t f \mathrm{d}t \mathrm{d}\mu\\
      &=& \int_0^{\infty} \int_E \frac12\lambda' |\nabla Q_{\lambda} h|^2 P_t f + \nabla Q_{\lambda} h \cdot \nabla P_t f \mathrm{d}\mu \mathrm{d}t\\
      &\leqslant& \int_0^{\infty}  -\frac{\mathrm{I}_\mu(P_tf)}{2\lambda'} \mathrm{d}t.
   \end{eqnarray*}
Using the Kantorovich dual (see \cite[Section 9.2]{BGL}, \cite[Chapter 1]{Villani1}) yields for $\nu = f\mu$
    \[ W_2^2(\nu, \mu) = 2 \sup\limits_{h} \left\{ \int_E Q_1h f\mathrm{d}\mu - \int_E h \mathrm{d}\mu \right\} \leqslant  \int_0^{\infty}  -\frac{\mathrm{I}_\mu(P_tf)}{\lambda'} \mathrm{d}t. \]

It is flexible to choose a nice $\lambda$ to prove Theorem \ref{Thm1}. For instance, if $\sqrt{\mathrm{Ent}_\mu(P_tf)}$ is integrable on $[0, \infty)$, let $\lambda(t) = \frac{\int_t^{\infty} \sqrt{\mathrm{Ent}_\mu(P_tf)}\mathrm{d}t}{\int_0^{\infty} \sqrt{\mathrm{Ent}_\mu(P_tf)}\mathrm{d}t}$, then it follows
   \begin{eqnarray}
     W_2^2(\nu, \mu) &\leqslant& \int_0^{\infty} \frac{\mathrm{I}_\mu(P_tf)}{\sqrt{\mathrm{Ent}_\mu(P_tf)}} \mathrm{d}t \; \cdot \; \int_0^{\infty} \sqrt{\mathrm{Ent}_\mu(P_tf)}\mathrm{d}t \nonumber\\
        &=& 2\sqrt{\mathrm{Ent}_\mu(f)} \int_0^{\infty} \sqrt{\mathrm{Ent}_\mu(P_tf)}\mathrm{d}t. \label{eq2-basic}
    \end{eqnarray}
We will revisit (\ref{eq2-basic}) in Section \ref{PT1} by means of derivative estimate of Wasserstein distance.

\begin{proof} It consists of two parts.

\textbf{Part 1}. First of all, using the inequality $\log x \leqslant x-1$ yields that
   \[ \mathrm{Ent}_\mu(f) = \int f\log\frac{f}{\mu(f)} \mathrm{d}\mu \leqslant \int f \cdot \frac{f-\mu(f)}{\mu(f)} \mathrm{d}\mu = \frac{1}{\mu(f)} \mathrm{Var}_\mu(f). \]
For $\mu(f) = 1$, we have $\mathrm{Ent}_\mu(f) \leqslant \mathrm{Var}_\mu(f)$. If PI holds with a constant $C_P$, we have further
    \[ \mathrm{Ent}_\mu(P_tf) \leqslant \mathrm{Var}_\mu(P_tf) \leqslant e^{-\frac{2}{C_P}t} \mathrm{Var}_\mu(f), \]
and then $\sqrt{\mathrm{Ent}_\mu(P_tf)}$ is integrable on $[0, \infty)$. It follows from (\ref{eq2-basic}) that
   \begin{eqnarray*}
     W_2^2(\nu, \mu) &\leqslant&  2\sqrt{\mathrm{Ent}_\mu(f)} \int_0^{\infty} \sqrt{\mathrm{Ent}_\mu(P_tf)}\mathrm{d}t\\
        &\leqslant& 2\sqrt{\mathrm{Ent}_\mu(f)} \int_0^{\infty} e^{-\frac{1}{C_P}t}\sqrt{\mathrm{Var}_\mu(f)}\mathrm{d}t \ =\  2C_P \sqrt{\mathrm{Ent}_\mu(f)} \sqrt{\mathrm{Var}_\mu(f)}. 
    \end{eqnarray*}
    
Inversely, assume there exists some $C>0$ such that
   \begin{eqnarray}
       W_2^2(\nu, \mu) \leqslant 2C \sqrt{\mathrm{Ent}_\mu(f)} \sqrt{\mathrm{Var}_\mu(f)}. \label{eq2-001}
   \end{eqnarray}
Various perturbation techniques give PI with a constant $\sqrt{2}C$, see \cite{Ledoux18, Villani2} and the references therein. For completeness, we write down a sketch.

Let $h$ be Lipschitz and bounded with $\mu(h) = 0$. Let $f_t = 1+\lambda th$ for $t\approx 0$ and some parameter $\lambda>0$. It follows from  (\ref{eq2-001}) that
     \[ 2\int Q_1(th) f_t\mathrm{d}\mu \leqslant  W_2^2(f_t\mu, \mu) \leqslant 2C \sqrt{\mathrm{Ent}_\mu(f_t)}\cdot \sqrt{\mathrm{Var}_\mu(f_t)}. \]
Substituting the Taylor's expansion $Q_1(th) = tQ_th = ht - \frac12 |\nabla h|^2 t^2 + o(t^2)$ at $t=0$ into the above inequality yields
    \begin{eqnarray}
      - \mu(\Gamma(h,h)) + 2\lambda \mu(h^2) \leqslant \sqrt{2} C\lambda^2 \mu(h^2), \label{eq2-004}
    \end{eqnarray}
which implies PI by taking $\lambda = \frac{\sqrt{2}}{2C}$. We obtain the equivalence between PI and  (\ref{eq2-001})  now.

\textbf{Part 2}. When we bound relative entropy by other functionals, it should lead to new types of transportation-variance inequalities. Indeed, for any  $p\geqslant 1$ holds by Jensen's inequality (recall $\mu(f) =1$ here)  that
   \begin{eqnarray*}
     \mathrm{Ent}_\mu(f) &=& \int f\log f \mathrm{d}\mu\\ 
     &\leqslant& \log \mu(f^2)
     = \log(\mathrm{Var}_\mu(f) + 1) \\
     &\leqslant& p \log((\mathrm{Var}_\mu(f))^{\frac1p}+1)
      \ \leqslant\ p (\mathrm{Var}_\mu(f))^{\frac1p}. 
    \end{eqnarray*}
If PI holds, it follows similarly from (\ref{eq2-basic}) 
   \begin{eqnarray*}
     W_2^2(\nu, \mu) &\leqslant& 2\sqrt{\mathrm{Ent}_\mu(f)} \int_0^{\infty} \sqrt{\mathrm{Ent}_\mu(P_tf)}\mathrm{d}t\\
        &\leqslant& 2p \mathrm{Var}_\mu^{\frac1{2p}}(f) \int_0^{\infty} \mathrm{Var}_\mu^{\frac1{2p}}(P_tf)\mathrm{d}t \ \leqslant\  2C_P p^2 \mathrm{Var}_\mu^{\frac1p}(f),
   \end{eqnarray*}
which covers the second inequality in Theorem \ref{Thm1} for $p=1$ and also gives the third one
   \[  W_2^2(\nu, \mu) \leqslant 2C_P \inf\limits_{p\geqslant 1} \left\{p^2 (\mathrm{Var}_\mu(f))^{\frac1p} \right\}. \]
   
Using PI again yields
   \[ W_2^2(\nu, \mu) \leqslant 2C_P \inf\limits_{p\geqslant 1} \left\{p^2 (\mathrm{Var}_\mu(f))^{\frac1p}\right\}  \leqslant 2C_P \inf\limits_{p\geqslant 1} \left\{p^2 {\big(} C_P\mu(\Gamma(f,f)) {\big)}^{\frac1p}\right\}, \]
which gives the fourth inequality in Theorem \ref{Thm1}.  It follows the fifth inequality by taking $p=1$ that
   \begin{eqnarray}
      W_2^2(\nu, \mu) \leqslant 2C_P^2 \mu(\Gamma(f,f)). \label{eq2-006}
   \end{eqnarray}
   
Inversely, still following the routine of perturbation technique, (\ref{eq2-006}) implies PI too. More precisely, recall the first part, we have a similar result as (\ref{eq2-004}) that
   \[  - \mu(\Gamma(h,h)) + 2\lambda \mu(h^2) \leqslant 2C_P^2 \lambda^2 \mu(\Gamma(h,h)),  \]
which implies PI with a constant $\sqrt{2} C_P$ by taking $\lambda = (\sqrt{2}C_P)^{-1}$.
\end{proof}

\bigskip

\section{Derivative of  quadratic Wasserstein distance along heat flow}
 \label{DHF}
 
In this section, we compute the derivative formula of $W_2(\nu_t, \mu)$ for $\frac{\mathrm{d}\nu_t}{\mathrm{d}\mu} = P_tf$. Recall that, in our notation, Otto-Villani \cite[Lemma 2]{Otto-Villani} (see \cite[Subsection 9.3.4]{Villani1} also) was actually concerned to the upper right-hand derivative of $W_2(\nu, \nu_t)$ and found a bound as
   \begin{eqnarray} 
      \frac{\mathrm{d}}{\mathrm{d}t}^+ W_2(\nu, \nu_t)  \leqslant   \limsup\limits_{s\to 0+}  W_2(\nu_t, \nu_{t+s})/s  \leqslant \sqrt{\mathrm{I}_\mu(P_tf)},  \label{eq1-1} 
   \end{eqnarray}
provided that $V\in C^2(\mathbb{R}^n)$ and $\mathrm{D}^2 V \geqslant \rho \mathrm{I}$ for some $\rho\in \mathbb{R}$ (namely the curvature-dimension condition $CD(\rho,\infty)$). The difference between $W_2(\nu_t, \mu)$ and $W_2(\nu, \nu_t)$ is that the former might be integrable for $t\in [0, +\infty)$.

According to \cite[Exercise 2.36]{Villani1}, there exists $h_t\in L^1(\mu)$ such that $\mu(h_t) =0$ and $Q_1 h_t  \in L^1(\nu_t)$, and the conjugate pair $(Q_1h_t, h_t)$ attains the supremum as
   \begin{eqnarray}
     W_2^2(\nu_t, \mu) = 2\sup\limits_{\mu(\phi)=0}\int Q_1\phi \mathrm{d}\nu_t  = 2\int Q_1h_t \mathrm{d}\nu_t = 2 \int Q_1 h_t P_t f \mathrm{d}\mu.  \label{eq1-3}
    \end{eqnarray}

Given nice initial data, we obtain the derivative formula for $W_2^2(\nu_t,\mu)$ in almost all $t$ with no condition on curvature. 
\begin{lemma} \label{Lem1}
Assume $f \in \mathbb{D}(\mathrm{L})$ has a positive lower bound. Assume $\mathrm{L}f$  is bounded. Then for almost all $t>0$, there exists some $h_t \in L^1(\mu)$ satisfying (\ref{eq1-3}) and
  \[  \frac{\mathrm{d}}{\mathrm{d}t} W_2^2(\nu_t, \mu) = 2 \int Q_1h_t \; \mathrm{L}P_t f \mathrm{d}\mu. \]
Moreover $|\frac{\mathrm{d}}{\mathrm{d}t} W_2^2(\nu_t, \mu)| \leqslant 2W(\nu_t, \mu)\sqrt{\mathrm{I}_\mu(P_tf)}$.
\end{lemma}
\begin{proof}
It consists of four steps. Note that $L^1(\nu_t) \subset L^1(\mu)$ in our case since $f$ has a positive lower bound and then $\nu_t(|h|) \geqslant \inf f \cdot \mu(|h|)$. The assumption of $\mathrm{L}f \in L^\infty(E)$ is reasonable due to that the resolvent operator $R_\lambda$ sends $C_\textrm{b}(E)$ into $C_\textrm{b}(E) \cap \mathbb{D}(\mathrm{L})$ and $\mathrm{L} = - R_\lambda^{-1}+ \lambda\mathrm{I}$ (see for example Evans \cite[Subsection 7.4.1]{Evans}).

\textbf{Step 1}. To show the continuity of $W_2(\nu_t,\mu)$ in $t$. 

Using the control inequality of weighted total variation (see \cite[Proposition 7.10]{Villani1}) yields that for any $t, t'>0$
   \begin{eqnarray*}
     W_2^2(\nu_{t'}, \nu_t) &\leqslant& 2\int d^2(x_0, \cdot)\left| P_{t'}f- P_tf \right|  \mathrm{d}\mu\\
         &=& 2\int d^2(x_0, \cdot)\left| \int_t^{t'} \mathrm{L}P_sf\mathrm{d}s\right|  \mathrm{d}\mu
         \leqslant 2|t'-t| \cdot ||\mathrm{L}f||_\infty \cdot \mu(d^2(x_0,\cdot)).
   \end{eqnarray*}
It follows from the triangle inequality $\left|W_2(\nu_{t'}, \mu) - W_2(\nu_{t}, \mu)\right| \leqslant W_2(\nu_{t'}, \nu_{t})$ that $W_2(\nu_t,\mu)$ is continuous in $t$. 

\textbf{Step 2}. To choose a conjugate pair $(Q_1h_t, h_t)$ satisfying (\ref{eq1-3}) and some auxiliary ``maximality" (which will be introduced in (\ref{eq1-223}) and applied for next step). 

First of all, let $(Q_1\tilde{h}_t, \tilde{h}_t) \in L^1(\nu_t) \times L^1(\mu)$ satisfy $\mu(\tilde{h}_t) =0$ and
    \[ W_2^2(\nu_t, \mu) = 2\int Q_1 \tilde{h}_t \mathrm{d}\nu_t.  \]
$Q_1 \tilde{h}_t$ may not have a gradient, so we take a sequence of bounded Lipschitz functions $\{\tilde{h}_{k,t}\}_{k\in\mathbb{N}}$ such that $\mu(\tilde{h}_{k,t})=0$ and $(Q_1\tilde{h}_{k,t}, \tilde{h}_{k,t})$ tends to $(Q_1\tilde{h}_t, \tilde{h}_t)$ in $L^1(\nu_t) \times L^1(\mu)$ as $k \to \infty$. Then $Q_1\tilde{h}_{k,t}$ is bounded Lipschitz too (see \cite[Subsection 3.3.2]{Evans}), and there exists $u_k\in [0,1]$ such that 
   \begin{eqnarray}
      \int Q_1{\big (}(1-u_k)\tilde{h}_{k,t}{\big )} \mathrm{d}\nu_t = \sup\limits_{0\leqslant u \leqslant 1} \int Q_1{\big (}(1-u) \tilde{h}_{k,t}{\big )} \mathrm{d}\nu_t. \label{eq1-223}
   \end{eqnarray}
Denote $h_{k,t} = (1-u_k)\tilde{h}_{k,t}$.  

Without loss of generality, assume $u_\infty=\lim\limits_{k\to\infty} u_k \in [0,1]$, denote
   \begin{eqnarray}
       h_t := (1-u_\infty)\tilde{h}_t = \lim\limits_{k\to\infty} (1-u_k)\tilde{h}_{k,t}  = \lim\limits_{k\to\infty} h_{k,t} \in L^1(\mu).\label{eq-equal-or-not}
    \end{eqnarray}
We want to show that $(Q_1h_t, h_t)$ is also a conjugate pair satisfying $W_2^2(\nu_t, \mu) = 2\int Q_1 h_t \mathrm{d}\nu_t$. The difference between $(Q_1h_t, h_t)$ and $(Q_1\tilde{h}_t, \tilde{h}_t)$ is that the former can be approximated by a special sequence of bounded Lipschitz pairs with the property (\ref{eq1-223}).

To this end, by the definition of infimum convolution, we have first
   \[ h_{k,t} \geqslant Q_1h_{k,t}  = (1-u_k) Q_{1-u_k} \tilde{h}_{k,t} \geqslant (1-u_k) Q_1 \tilde{h}_{k,t}, \]
which means that $Q_1h_{k,t}$ falls between two $L^1$-convergent sequences. By virtue of the Prokhorov theorem (namely the tightness argument) together with the fact of  $L^1(\nu_t) \subset L^1(\mu)$, one can extract a subsequence of $Q_1h_{k,t}$ (denoted by itself for the ease of notation) converging in $L^1(\nu_t)$. Denote $\phi_t = \lim\limits_{k \to \infty} Q_1h_{k,t}$, which satisfies 
   \[ \phi_t(x) - h_t(y) \leqslant \frac12 d^2(x,y) \]
almost everywhere and then $\phi_t(x) \leqslant Q_1h_t(x)$ and (since $\mu(h_t) =0$) 
   \[ 2\nu_t(\phi_t) \leqslant 2\int Q_1h_t \mathrm{d}\nu_t \leqslant W_2^2(\nu_t,\mu).\] 
On the other hand, due to the definition of $h_{k,t}$ in (\ref{eq1-223}), it follows 
   \[ 2\nu_t(\phi_t) =\lim\limits_{k\to \infty} 2\nu_t(Q_1 h_{k,t}) \geqslant \lim\limits_{k\to \infty} 2\nu_t(Q_1 \tilde{h}_{k,t}) = W_2^2(\nu_t,\mu). \]
Hence, $(\phi_t, h_t)$ attains the supremum of the dual Kantorovich problem too. Moreover, it follows $\phi_t = Q_1 h_t$ almost everywhere with respect to $\nu_t$ and $\mu$ as well since $f$ has a positive lower bound.

\textbf{Step 3}. To estimate upper and lower derivatives of  $W_2^2(\nu_t,\mu)$. 

For $(Q_1 h_t,h_t)$, we have 
   \begin{eqnarray}
    \underline{D}^+_t &:=& \liminf\limits_{s\to 0+} \frac{W_2^2(\nu_{t+s}, \mu) - W_2^2(\nu_{t}, \mu) }{s} \nonumber\\
     &\geqslant&  \lim\limits_{s\to 0+} \frac{2}{s}\left(\int Q_1h_t \mathrm{d}\nu_{t+s} - \int Q_1h_t \mathrm{d}\nu_t\right) \ =\ 2 \int Q_1h_t \; \mathrm{L}P_t f \mathrm{d}\mu. \label{eq1-221}
   \end{eqnarray}
Similarly, we have
   \begin{eqnarray}
     \overline{D}^-_t &:=& \limsup\limits_{s\to 0+} \frac{W_2^2(\nu_{t}, \mu) - W_2^2(\nu_{t-s}, \mu) }{s} \nonumber\\
     &\leqslant&  \lim\limits_{s\to 0+} \frac{2}{s}\left(\int Q_1h_t \mathrm{d}\nu_t - \int Q_1h_t \mathrm{d}\nu_{t-s}\right) \ =\  2 \int Q_1h_t \; \mathrm{L}P_t f \mathrm{d}\mu. \label{eq1-222}
    \end{eqnarray}

Recall the approximating sequence $(Q_1 h_{k,t}, h_{k,t})$ for $(Q_1 h_t, h_t)$ in Step 2, using the formula of integration by parts and the H\"older inequality yields that
   \[
     \left|\int Q_1h_{k, t} \; \mathrm{L}P_t f \mathrm{d}\mu \right| = \left|\int \nabla Q_1h_{k, t} \; \nabla P_t f \mathrm{d}\mu \right| \leqslant \sqrt{\int |\nabla Q_1h_{k, t}|^2 \mathrm{d}\nu_t} \cdot \sqrt{\mathrm{I}_\mu(P_tf)}. 
    \]
Since $Q_sh_{k, t}$ solves the Hamilton-Jacobi equation $\frac{\mathrm{d}}{\mathrm{d}s} Q_sh_{k, t} + \frac12 |\nabla Q_sh_{k, t}|^2 = 0$ (see \cite[Subsection 3.3.2]{Evans}),  we have by (\ref{eq1-223}) (namely the integral ``maximality" for $Q_1h_{k, t}$) that
   \begin{eqnarray*}
     \int |\nabla Q_1h_{k, t}|^2 \mathrm{d}\nu_t &=& \lim\limits_{u\to  0+}  2\int \frac{Q_{1-u}h_{k, t} - Q_1h_{k, t}}{u} \mathrm{d}\nu_t \\
     &=& \lim\limits_{u\to  0+}  2\int \frac{\frac{1}{1-u}Q_1{\big(}(1-u)h_{k, t}{\big)} - Q_1h_{k, t}}{u} \mathrm{d}\nu_t \\
     &\leqslant& \lim\limits_{u\to  0+}  2\cdot \frac{\frac{1}{1-u} - 1}{u} \cdot \int Q_1h_{k,t} \mathrm{d}\nu_t \ \leqslant \ W_2^2(\nu_t,\mu),
   \end{eqnarray*}
which implies by taking $k \to \infty$
   \begin{eqnarray}
       2\left|\int Q_1h_t \; \mathrm{L}P_t f \mathrm{d}\mu\right| = \lim\limits_{k\to +\infty} 2\left|\int Q_1h_{k,t} \; \mathrm{L}P_t f \mathrm{d}\mu\right| \leqslant 2W_2(\nu_{t}, \mu)\sqrt{\mathrm{I}_\mu(P_tf)} =: A_t. \label{eq1-4}
   \end{eqnarray}
Note that $A_t$ is continuous in $t$.
  
\textbf{Step 4}. To show the Lipschitz property of $W_2^2(\nu_t,\mu)$. 

For convenience, denote $F(t) = W_2^2(\nu_t,\mu)$. Heuristically, using (\ref{eq1-221}) and (\ref{eq1-4}) yields a local estimate that for any $t>0$ there exists $s>0$ such that $F(t+s) - F(t) \geqslant -O(s)$. It follows $F(b) - F(a) \geqslant -O(b-a)$ for any interval $[a,b]\subset\mathbb{R}^+$ if one could ``find" a finite partition of $[a,b]$ and sum up all the local estimates. Similarly, using (\ref{eq1-222}) and (\ref{eq1-4}) yields $F(b) - F(a) \leqslant O(b-a)$, and then gives the local Lipschitz property.

The rest of the proof is basically a careful application of Borel-Lebesgue covering theorem. Fix arbitrary $\varepsilon >0$. Let $K = \sup\limits_{t\in [a,b]}A_t+\varepsilon$. For any $t\in [a,b]$, there exists some $\eta_t\in(0,b-a]$ by using (\ref{eq1-221}) and (\ref{eq1-4}) such that for all $s\in (0,\eta_t]$
    \[ F(t+s) - F(t) > s \left( 2 \int Q_1h_t \; \mathrm{L}P_t f \mathrm{d}\mu - \varepsilon\right) \geqslant -s(A_t+\varepsilon) \geqslant -sK \geqslant -\eta_t K. \]
On the other hand, the continuity of $F(t)$ implies there exists $\tilde{\eta}_t\in (0, \eta_t]$ such that for all $-s\in [-\tilde{\eta}_t, 0]$
     \[ |F(t) - F(t-s)| <  \eta_t K. \]
Then the open interval $I_t = (t-\tilde{\eta}_t, t+\eta_t)$ is of length no less than $\eta_t$ and no more than $2\eta_t$, and holds for any $t_2\geqslant t \geqslant t_1$ or $t\geqslant t_2 \geqslant t_1$ in $I_t$
     \begin{eqnarray}
         F(t_2) - F(t_1) > -2\eta_t K \geqslant -2|I_t|K. \label{eq-lower-Lip}
    \end{eqnarray}
(Notice that we don't know whether (\ref{eq-lower-Lip}) is true for $t_2\geqslant t_1> t$.)
 
The collection of all $I_t$ becomes an open covering of $[a,b]$, which implies a finite sub-covering $\mathcal{I}$. To reduce overlaps, we have to do some selection. Starting from $t_0 = a$, one can successively take the $i$-th open interval $I_{t_i}$ from $\mathcal{I}$ for $i=1, 2\ldots$ satisfying next two properties:
\begin{enumerate}
\item[(1).]  $I_{t_i} \cap I_{t_{i-1}} \neq \emptyset$, and $I_{t_i}$ contains the right-hand endpoint of $I_{t_{i-1}}$.
\item[(2).] If there is another $I_{t_*} \in \mathcal{I}$ intersecting with $I_{t_{i-1}}$, then $I_{t_*} \subset \bigcup\limits_{j\leqslant i} I_{t_j}$, namely the right-hand endpoint of $I_{t_*}$ doesn't exceed $I_{t_i}$. It means $I_{t_i}$ is the most effective cover than any other $I_{t_*}$.
\end{enumerate}
This procedure will stop at time $N$ once $I_{t_N}$ contains $b$.  

Now, we have a chain $I_{t_0}, I_{t_1}, \ldots, I_{t_N}$ satisfying that each element only intersects with its neighbors, which means their overlap is at most $2$-fold for every point in $[a,b]$. Let $t_{i-1, i} \in I_{t_{i-1}} \cap I_{t_i}$ satisfy $t_{i-1, i} \leqslant t_i$  for $i = 1, \dots, N$ and $a\leqslant t_{0,1} \leqslant t_{1,2} \cdots \leqslant t_{N-1,N} \leqslant b$. It must occur either $t_{i-1, i} \leqslant t_i \leqslant t_{i, i+1}$ or $t_{i-1, i} \leqslant t_{i, i+1} \leqslant t_i$ for each $i$. In any case, we obtain an interpolation by (\ref{eq-lower-Lip})
      \begin{eqnarray*}
       F(b) - F(a)    &=& F(b) - F(t_{N-1,N}) + \sum\limits_{i=1}^{N-1} F(t_{i,i+1}) - F(t_{i-1,i}) + F(t_{0,1}) - F(a)\\
       &\geqslant& -2|I_{t_N}|K - \sum\limits_{i=1}^{N-1} 2|I_{t_{i}}|K - 2|I_{t_0}|K \ \geqslant\ - 8(b-a)K.
      \end{eqnarray*}

Similarly, it follows from (\ref{eq1-222}) and (\ref{eq1-4}) that
    \[ F(b) - F(a) \leqslant  8(b-a)K.\]

Combining the above estimates yields that $F(t)=W_2^2(\nu_t,\mu)$ is locally Lipschitz and then has a derivative for almost all $t>0$ as
  \[ \frac{\mathrm{d}}{\mathrm{d}t} W_2^2(\nu_t, \mu) = 2 \int Q_1h_t \; \mathrm{L}P_t f \mathrm{d}\mu. \]
It follows that for almost all $t>0$
   \[ \left|\frac{\mathrm{d}}{\mathrm{d}t} W_2^2(\nu_t, \mu)\right| \leqslant 2W_2(\nu_{t}, \mu)\sqrt{\mathrm{I}_\mu(P_tf)}, \]
which can be rewritten to 
   \[ \left|\frac{\mathrm{d}}{\mathrm{d}t} W_2(\nu_t, \mu)\right| \leqslant \sqrt{\mathrm{I}_\mu(P_tf)}. \]
The proof is completed.
\end{proof}

\begin{remark}
It is interesting to ask further that whether $h_t = \tilde{h}_t$ almost everywhere (namely $u_\infty = 0$ in (\ref{eq-equal-or-not})). For any positive $\alpha$ and $\beta$ with $\alpha+ \beta = 1$, we have $\alpha Q_1\tilde{h}_t + \beta Q_1  h_t \leqslant Q_1 {\big(}\alpha \tilde{h}_t + \beta h_t {\big)}$ and
   \begin{eqnarray*}
    \ \ \ \ W_2^2(\nu_t,\mu) = 2\int \alpha Q_1\tilde{h}_t + \beta Q_1h_t \mathrm{d}\nu_t \leqslant  2\int Q_1 {\big(}\alpha \tilde{h}_t + \beta h_t {\big)} \mathrm{d}\nu_t \leqslant W_2^2(\nu_t,\mu),
   \end{eqnarray*} 
which implies $\alpha Q_1\tilde{h}_t + \beta Q_1h_t = Q_1 {\big(}\alpha \tilde{h}_t + \beta h_t {\big)}$ almost everywhere. It follows that for almost every $x\in E$ and $h =\tilde{h}_t$ or $h_t$ or $\alpha \tilde{h}_t + \beta h_t$, $Q_1h(x)$ can take its value at the same critical point $y_x$ such that $Q_1h(x) = h(y_x) + \frac12d^2(x,y_x)$  (or the same point sequence $\{y_x^{(n)}\}$ such that $Q_1h(x) = \lim\limits_{n\to +\infty} h(y_x^{(n)}) + \frac12d^2(x,y_x^{(n)})$). If $u_\infty \neq 0$ and $\tilde{h}_t$ is bounded and differentiable, we have $\nabla h_t(y_x) = \nabla \tilde{h}_t(y_x) = x - y_x$ and then $\nabla h_t(y_x) = \nabla \tilde{h}_t(y_x)  \equiv 0$  since $h_t = (1-u_\infty)\tilde{h}_t$, which means $\tilde{h}_t$ has to be a constant function and furthermore $\tilde{h}_t \equiv 0$ for $\mu(\tilde{h}_t) = 0$. This suggests that $h_t = \tilde{h}_t$ is true, however, it seems complicated to deal with $L^1$ functions.
\end{remark}

The same argument is also effective in reproving Lemma 2 in \cite{Otto-Villani} as
   \[ \left|\frac{\mathrm{d}}{\mathrm{d}t} W_2(\nu, \nu_t)\right| \leqslant \sqrt{\mathrm{I}_\mu(P_tf)}, \]
which avoids using the second inequality in (\ref{eq1-1}).

 \bigskip

\section{The second proof of Theorem \ref{Thm1}}
 \label{PT1}
 
 \begin{proof}  Assume PI holds with a constant $C_P$. Recall that
   \[ \mathrm{Ent}_\mu(P_tf) \leqslant \mathrm{Var}_\mu(P_tf) \leqslant \exp\{-2t/C_P\} \mathrm{Var}_\mu(f), \]
which implies $\mathrm{Ent}_\mu(P_tf) \to 0$ for $t\to \infty$. Using the same method in the second part of \cite[Lemma 3]{Otto-Villani} yields $W_2(\nu_t, \mu) \to 0$ too. More precisely, $W_2(\nu_t, \mu)$ decays exponentially fast due to that for any continuous $\xi$ with $|\xi(x)| \leqslant C(d^2(x_0, x) +1)$,
   \begin{eqnarray*}
      \left| \int \xi \mathrm{d}\nu_t - \int \xi \mathrm{d}\mu \right| 
      &\leqslant& C\int |P_tf - 1|(d^2(x_0, \cdot) +1)d\mu \\
       &\leqslant& C\sqrt{\mathrm{Var}_\mu(P_tf)} \sqrt{\mu((d^2(x_0, \cdot) +1)^2)}, 
    \end{eqnarray*}
where the integrability of $d^4(x_0,\cdot)$ comes from PI as well.
 
For simplicity, assume $f$ fulfills all the conditions in Lemma \ref{Lem1}, then we have by using the H\"older inequality to get (\ref{eq2-basic}) again
     \begin{eqnarray*} 
        W_2^2(\nu, \mu) &=& \left(\int_0^{\infty} \frac{\mathrm{d}}{\mathrm{d}s}W_2(\nu_s, \mu) \mathrm{d}s \right)^2
          \ \leqslant\  \left(\int_0^{\infty} \sqrt{\mathrm{I}_\mu(P_sf)} \mathrm{d}s \right)^2 \\
          &=&  \left(\int_0^{\infty} \frac{\sqrt{\mathrm{I}_\mu(P_sf)}}{\sqrt[4]{\mathrm{Ent}_\mu(P_sf)}} \cdot \sqrt[4]{\mathrm{Ent}_\mu(P_sf)} \,\mathrm{d}s \right)^2 \ \leqslant\  2\sqrt{\mathrm{Ent}_\mu(f)} \int_0^{\infty} \sqrt{\mathrm{Ent}_\mu(P_tf)}\mathrm{d}t.
      \end{eqnarray*}
The following steps are the same as those in Section \ref{FPT}.
\end{proof}

Alternatively, using Lemma \ref{Lem1} and H\"older inequality yields also for any $t>0$
    \begin{eqnarray}
      W_2^2(\nu_t, \mu) = \int_ t^{\infty} \frac{\mathrm{d}}{\mathrm{d}s}W_2^2(\nu_s, \mu) \mathrm{d}s  &\leqslant& 2\int_t^{\infty}  W_2(\nu_s, \mu) \sqrt{\mathrm{I}_\mu(P_sf)} \mathrm{d}s
        \nonumber\\
       &\leqslant& 2 \sqrt{\int_t^{\infty}  W_2^2(\nu_s, \mu)\mathrm{d}s} \cdot  \sqrt{\int_t^{\infty} \mathrm{I}_\mu(P_sf) \mathrm{d}s} \nonumber\\
       &=& 2 \sqrt{\int_t^{\infty}  W_2^2(\nu_s, \mu)\mathrm{d}s} \cdot \sqrt{\mathrm{Ent}_\mu(P_t f)} \label{eq2-Ent}
    \end{eqnarray}
(\ref{eq2-Ent}) looks like (\ref{eq2-basic}), which is still useful to prove Theorem \ref{Thm1} as follows.
 
Denote $\mathcal{W}_t = \sqrt{\int_t^{\infty}  W_2^2(\nu_s, \mu)\mathrm{d}s}$ (it is finite since $W_2(\nu_t, \mu)$ decays exponentially fast), (\ref{eq2-Ent}) can be rewritten to
   \[ -\frac{\mathrm{d}}{\mathrm{d}t} \mathcal{W}_t \leqslant \sqrt{\mathrm{Ent}_\mu(P_tf)} \leqslant \sqrt{\mathrm{Var}_\mu(P_tf)} \leqslant \exp\{-t/C_P\}\sqrt{\mathrm{Var}_\mu(f)}.  \]
and then
   \begin{eqnarray*}
     \mathcal{W}_t  =  \int_t^\infty -\frac{\mathrm{d}}{\mathrm{d}s} \mathcal{W}_s \mathrm{d}s &\leqslant& \int_t^\infty \exp\{-s/C_P\} \mathrm{d}s\sqrt{\mathrm{Var}_\mu(f)} \\
        &=& C_P\exp\{-t/C_P\} \sqrt{\mathrm{Var}_\mu(f)}. 
    \end{eqnarray*}
Substituting this estimate back to (\ref{eq2-Ent}) for $t=0$ gives us
   \begin{eqnarray*}
       W_2^2(\nu, \mu) &\leqslant& 2 \sqrt{\int_0^{\infty}  2C_P\mathrm{Var}_\mu(P_sf) \mathrm{d}s} \cdot \sqrt{\mathrm{Ent}_\mu(f)} \nonumber\\
       &\leqslant& 2 \sqrt{\int_0^{\infty}  2C_P\exp\left(-\frac{2}{C_P}s\right)\mathrm{Var}_\mu(f) \mathrm{d}s} \cdot \sqrt{\mathrm{Ent}_\mu(f)} \nonumber\\
       &=& 2C_P \sqrt{\mathrm{Var}_\mu(f)}\cdot \sqrt{\mathrm{Ent}_\mu(f)}.
    \end{eqnarray*}
The following steps are the same as before.

By the way, if one is concerned to the quantity $W_2^2(\tilde{\nu}_t,\mu)$ for $\frac{\mathrm{d}\tilde{\nu}_t}{\mathrm{d}\mu} = \frac{|P_t\sqrt{f}|^2}{\mu(|P_t\sqrt{f}|^2)}$, it also decays exponentially fast provided that PI holds. Firstly we have for any $g^2\mu\in \mathcal{P}(E)$ (denote $m= \mu(g)$ and $\sigma_t^2 = \mu\left((P_t g-m)^2\right)$)
   \begin{eqnarray*}
      \mathrm{Var}_\mu(g^2) \leqslant \int |g^2-m^2|^2\mathrm{d}\mu \leqslant 2\int |g-m|^4 \mathrm{d}\mu + 8m^2\int |g-m|^2\mathrm{d}\mu.         
   \end{eqnarray*}
Then it follows from PI that
   \begin{eqnarray*}
      \frac{\mathrm{d}}{\mathrm{d}t} \mu\left( \left(P_tg - m\right)^4 \right) 
      &=& -12\mu\left( \left(P_tg-m\right)^2 \left|\nabla P_tg\right|^2 \right)\\
      &\leqslant& -3C_P^{-1}\mu\left(\left( \left(P_t g-m\right)^2   - \sigma_t^2 \right)^2\right)\\
      &=& -3C_P^{-1}\left[ \mu\left(\left(P_tg - m\right)^4 \right) - \sigma_t^4 \right],
   \end{eqnarray*}
and 
   \[ \frac{\mathrm{d}}{\mathrm{d}t} \sigma_t^4 = -4\sigma_t^2 \mu\left( \left| \nabla P_tg \right|^2 \right)
         \leqslant -4C_P^{-1}\sigma_t^4. \]
Set $\Lambda_t = \mu\left( \left(P_tg-m\right)^4 \right) + \lambda \sigma_t^4$ with the parameter $\lambda$, we have
   \[ \frac{\mathrm{d}}{\mathrm{d}t} \Lambda_t \leqslant C_P^{-1} \left(-3\Lambda_t  + (3 - \lambda) \sigma_t^4\right), \]
which implies by taking $\lambda=3$ that
   \[ \frac{\mathrm{d}}{\mathrm{d}t} \Lambda_t \leqslant -3C_P^{-1}\Lambda_t \]
and then $\Lambda_t \leqslant \exp\left(-3t/C_P\right) \Lambda_0$. 

Hence using Theorem \ref{Thm1} yields for $g= \sqrt{f}$ that
   \begin{eqnarray*}
      W_2^2\left(\tilde{\nu}_t, \mu \right) &\leqslant& 2C_P \mathrm{Var}_\mu(\frac{\mathrm{d}\tilde{\nu}_t}{\mathrm{d}\mu}) 
         \ \leqslant\   \frac{2C_P}{\left(\mu(|P_t g|^2)\right)^2}  \mathrm{Var}_\mu \left((P_t g)^2\right)\\
         &\leqslant&  \frac{4C_P}{m^4} (\Lambda_t + 4m^2 \sigma_t^2)
         \ \leqslant\ \frac{4C_P}{m^4}(e^{-3t/C_P}\Lambda_0 + e^{-2t/C_P}4m^2\sigma_0^2),  
   \end{eqnarray*}
where the total rate is no more than $e^{-2t/C_P}$.

 \bigskip
 
 \section{The logarithmic Sobolev inequality and strict contraction of heat flow in Wasserstein space}
 \label{ELSI}
 
In this section, we prove Proposition \ref{Prop1}. The curvature-dimension condition plays a fundamental role such that we can compare several functionals for heat flow at different times. The derivative estimate in previous section is also useful.
 
 \begin{proof} Assume $V$ is a smooth potential satisfying the curvature-dimension condition $CD(\rho,\infty)$.
 
 If the LSI holds, it is known that the entropy along heat flow decays exponentially fast. Moreover, the Talagrand inequality comes true (see \cite{Otto-Villani} or \cite[Theorem 9.6.1]{BGL}), namely for any positive bounded $f$ and any $t > T>0$
    \[ W_2^2(P_tf \mu,\mu) \leqslant 2C_{LS} \mathrm{Ent}_\mu(P_t f) \leqslant 2C_{LS} e^{-2(t-T)/C_{LS}} \mathrm{Ent}_\mu(P_T f). \]
On the other hand, based on the logarithmic Harnack inequlity (see \cite[Remark 5.6.2]{BGL})
  \[ P_T(\log f)(x) \leqslant \log P_Tf(y) +  \frac{\rho d(x,y)^2}{2(e^{2\rho T} -1)}, \]
it follows from the the same argument as \cite[Page 446]{BGL} that
   \[ \mathrm{Ent}_\mu(P_T f) \leqslant \frac{1}{2\beta(T)} W_2^2(f \mu,\mu),  \]
where $\frac{1}{\beta(T)} = \frac{\rho}{1-e^{-2\rho T}} - \rho \ (=\frac{1}{2T} \ \textrm{for}\ \rho = 0)$. Combining the above estimates yields
   \[  W_2^2(P_tf \mu,\mu) \leqslant  \gamma(T) e^{-2t/C_{LS}} W_2^2(f \mu,\mu) \]
by letting $\gamma(T) = \frac{C_{LS} e^{2T/C_{LS}}}{\beta(T)}$, which attains its minimum at $T_0=\frac{1}{2|\rho|}\log (1+C_{LS}|\rho|)$. So now we obtain the exponential decay for $t > T_0$. 

For $0<t\leqslant T_0$, there is a general bound according to the heat flow contraction in Wasserstein space (see \cite[Theorem 9.7.2]{BGL}) as
   \[ W_2^2(P_tf \mu,\mu) \leqslant  e^{-2\rho t} W_2^2(f \mu,\mu) = e^{(2C_{LS}^{-1} - 2\rho)t} e^{-2 t/C_{LS}} W_2^2(f \mu,\mu). \]
Combining two regions gives us a control constant $C := \sqrt{\max\{\gamma(T_0), e^{(2C_{LS}^{-1} - 2\rho)T_0}, 1\}}$ such that for all $t>0$ and $\kappa :=C_{LS}^{-1}$
   \[ W_2(P_tf \mu,\mu) \leqslant  Ce^{-\kappa t} W_2(f \mu,\mu). \]
   
 Conversely, if $W_2(P_tf \mu,\mu) \leqslant Ce^{-\kappa t} W_2(f\mu,\mu)$, there exists $t$ (independent of $f$) such that $\eta := Ce^{-\kappa t} < 1$. Using the derivative estimate for nice $f$ (see Lemma \ref{Lem1}) yields
     \begin{eqnarray*}
       W_2^2(f \mu,\mu)
           &=& W_2^2(f \mu,\mu) - W_2^2(P_tf \mu,\mu) + W_2^2(P_tf \mu,\mu) \\
           &\leqslant& \int_0^t 2W_2(P_sf\mu,\mu) \sqrt{\mathrm{I}_\mu(P_sf)} \mathrm{d}s + \eta^2 W_2^2(f \mu,\mu) \\
           &\leqslant& 2\left( \int_0^t W_2^2(P_sf\mu,\mu)  \mathrm{d}s \right)^\frac12 \left( \int_0^t  \mathrm{I}_\mu(P_sf) \mathrm{d}s \right)^\frac12 + \eta^2 W_2^2(f \mu,\mu).           
     \end{eqnarray*}
Based on the heat flow contraction and information contraction (see \cite[Eq. 5.7.4]{BGL})
    \[ \mathrm{I}_\mu(P_sf)  \leqslant  e^{-2\rho s}\mathrm{I}_\mu(f),  \]
we have further
    \begin{eqnarray*}
      W_2^2(f \mu,\mu)
           &\leqslant& 2 \left( \int_0^t C^2e^{-2\kappa t}W_2^2(f\mu,\mu)  \mathrm{d}s \right)^\frac12 \left( \int_0^t  e^{-2\rho s} \mathrm{I}_\mu(f) \mathrm{d}s \right)^\frac12  \hspace*{-3mm} + \eta^2 W_2^2(f \mu,\mu)\\
           &\leqslant& (\varepsilon + \eta^2) W_2^2(f \mu,\mu) + \frac{C^2(1-e^{-2\kappa t})(1-e^{-2\rho t})}{4\kappa\rho\varepsilon} \mathrm{I}_\mu(f),
      \end{eqnarray*}
where the last step comes from the Cauchy-Schwarz inequality for any $\varepsilon>0$. It follows $\textrm{W}_2\textrm{I}$ by taking $\varepsilon = \eta = \frac12$  explicitly that
    \[  W_2^2(f \mu,\mu) \leqslant \frac{2C^2(1-e^{-2\rho t})}{\kappa\rho} \mathrm{I}_\mu(f). \]
Since $\textrm{W}_2\textrm{I}$ is equivalent to LSI under $CD(\rho,\infty)$ by virtue of the HWI inequality (see \cite{Otto-Villani} or \cite[Subsection 9.3]{BGL})
    \[ \mathrm{Ent}_\mu(f) \leqslant W_2(f\mu,\mu)\sqrt{\mathrm{I}_\mu(f)} - \frac{\rho}{2} W_2^2(f\mu,\mu), \]
we complete the proof.
 \end{proof}

 \bigskip

\section{Centralization of quadratic Wasserstein distance}
 \label{CWD}

Recall the notation $c=\mu(\sqrt{f})$ and $\sigma^2 = \mathrm{Var}_\mu(\sqrt{f})$, now we prove Theorem \ref{Thm2}. 

\begin{proof}
For any bounded Lipschitz $h$ with $\mu(h) =0$, let $m_t = \mu(Q_th)$, we have
   \begin{eqnarray*}
      \mu(Q_th f) &=& \int Q_th (\sqrt{f} -c)^2 \mathrm{d}\mu + 2c \int Q_th (\sqrt{f}-c)\mathrm{d}\mu + c^2 \int Q_th\mathrm{d}\mu\\
     &=& \int Q_th (\sqrt{f} -c)^2 \mathrm{d}\mu + 2c \int (Q_th - m_t) (\sqrt{f}-c)\mathrm{d}\mu + c^2m_t.
   \end{eqnarray*}
Taking  any interval $[a,b] \subset \mathbb{R}^+$ and any nonnegative $\phi \in C^1([a,b])$, we integrate both sides to get
  \begin{eqnarray*}
    &&\mathbb{I}_0 \ :=\ \int_a^b \mu(Q_th f)\phi \mathrm{d}t\\ 
    &=& \hspace*{-3mm} \int_a^b \mu(Q_th (\sqrt{f} -c)^2)\phi \mathrm{d}t +  2c \int_a^b \mu{\big(}(Q_th - m_t)(\sqrt{f}-c){\big)} \phi \mathrm{d}t+\ c^2 \int_a^b m_t\phi \mathrm{d}t. 
   \end{eqnarray*}
For convenience, denote the right-hand three terms by $\mathbb{I}_1, \mathbb{I}_2, \mathbb{I}_3$ respectively. Using the Cauchy-Schwarz, H\"older and Poincar\'e inequalities yields for any $\lambda > 0$
   \begin{eqnarray*}
       \mathbb{I}_2
       &=& 2c\int \left(\int_a^b (Q_th - m_t)\phi(t) \mathrm{d}t\right)  (\sqrt{f}-c)\mathrm{d}\mu \\
       &\leqslant& \lambda c^2 \int  \left(\int_a^b (Q_th - m_t)\phi(t) \mathrm{d}t\right)^2\mathrm{d}\mu + \frac1{\lambda}\mu((\sqrt{f}-c)^2) \\
       &\leqslant& \lambda c^2 (b-a) \int  \int_a^b (Q_th - m_t)^2 \phi^2(t) \mathrm{d}t \mathrm{d}\mu + \frac1{\lambda}\sigma^2\\
       &=& \lambda c^2 (b-a) \int_a^b \mu \left((Q_th - m_t)^2\right) \phi^2(t) \mathrm{d}t  + \frac1{\lambda}\sigma^2\\
       &\leqslant& \lambda c^2 (b-a) C_P \int_a^b \mu \left(|\nabla Q_th|^2\right) \phi^2(t) \mathrm{d}t  + \frac1{\lambda}\sigma^2\\
       &=& 2 \lambda c^2 (b-a) C_P \int \int_a^b - \frac{\mathrm{d}}{\mathrm{d}t} Q_th \; \phi^2(t) \mathrm{d}t\mathrm{d}\mu + \frac1{\lambda}\sigma^2,
   \end{eqnarray*}
where the last step comes from the Hamilton-Jacobi equation. Using the integration by parts gives
    \[ \int_a^b -  \frac{\mathrm{d}}{\mathrm{d}t} Q_th \; \phi^2(t) \mathrm{d}t =  Q_ah \phi^2(a) - Q_bh \phi^2(b) + \int_a^bQ_th \cdot 2\phi\phi'\mathrm{d}t. \]
If $\phi(a) = \phi(b) = 0$, we have further
   \[ \mathbb{I}_2 \leqslant 4 \lambda c^2 (b-a) C_P \int_a^b m_t \phi\phi' \mathrm{d}t  +  \frac1{\lambda}\sigma^2, \]
and then 
   \[ \mathbb{I}_2 + \mathbb{I}_3 \leqslant   c^2 \int_a^b m_t \phi \left[4 \lambda (b-a) C_P\phi' + 1\right]  \mathrm{d}t + \frac1{\lambda}\sigma^2. \]
   
Now we want to drop the first integral on the right side of above inequality. For instance, take $a= \frac12$, $b=1$, $\phi(t) = (t-a)(b-t)$ (satisfying $\phi(a) = \phi(b) = 0$, $\phi\geqslant 0$ and $|\phi'| \leqslant \frac12$), and $\lambda = C_P^{-1}$, then  for $t\in [a,b]$, the quantity 
   \[ \psi := \left(4 \lambda (b-a) C_P\phi' + 1\right) \geqslant 0,\]
which implies $ \int_a^b m_t \phi\psi  \mathrm{d}t \leqslant 0$ since the monotonicity of $Q_t$ in $t$ gives $m_t=\mu(Q_th) \leqslant \mu(h) = 0$. Hence $\mathbb{I}_2 + \mathbb{I}_3 \leqslant  C_P\sigma^2$.

Finally, combining all above estimates yields
    \[ \mathbb{I}_0 \leqslant \mathbb{I}_1 + C_P\sigma^2. \]
Denote $M = \int_a^b \phi\mathrm{d}t = \frac{1}{48}$, it follows  
   \[ M \cdot \mu(Q_bh f) \leqslant \mathbb{I}_0 \leqslant \mathbb{I}_1 + C_P\sigma^2 \leqslant M \cdot \mu(Q_ah (\sqrt{f}-c)^2) + C_P\sigma^2, \]
which implies by the Kantorovich dual of $W_2$-distance that
    \[ \frac{M}{2b} W_2^2(f\mu, \mu) \leqslant \frac{M}{2a} \sigma^2 W_2^2\left(\frac{(\sqrt{f}-c)^2}{\sigma^2}\mu,\mu\right) + C_P\sigma^2. \]
The proof is completed.
\end{proof}

When we check the proof, for any $\theta$ still holds
  \begin{eqnarray*}
    \mu(Q_th f) &=& \mu(Q_th (\sqrt{f} -\theta)^2) + 2\theta \mu{\big(}Q_th(\sqrt{f}-\theta){\big)} + \theta^2 \mu(Q_th)\\
    &=& \mu(Q_th (\sqrt{f} -\theta)^2) + 2\theta \mu{\big(}(Q_th-m_t)(\sqrt{f}-\theta){\big)} + (2\theta c - \theta^2) \mu(Q_th)\\
    &=& \mu(Q_th (\sqrt{f} -\theta)^2) + 2\theta \mu{\big(}(Q_th-m_t)(\sqrt{f}-c){\big)} + (2\theta c - \theta^2) \mu(Q_th).
   \end{eqnarray*}
Denote $\sigma_\theta^2 = \mu((\sqrt{f}-\theta)^2)$. Once $\theta\in (0, 2c)$, we have by the same argument
  \[ W_2^2(f\mu, \mu) \leqslant C_1(\theta) \sigma_\theta^2 W_2^2\left(\frac{(\sqrt{f}-\theta)^2}{\sigma_\theta^2}\mu,\mu\right) + C_2(\theta)C_P\sigma^2, \] 
where $C_1(\theta)$ and $C_2(\theta)$ are two constants depending on $\theta$.

\bigskip

\section{Application to quadratic transportation-information inequality}
 \label{ATI}
  
According to \cite{CG, CGW, Liuy}, the Lyapunov condition (\ref{eqLya}) implies that there are two constants $C_3, C_4>0$ such that
   \begin{eqnarray}
     \int d^2(x_0, \cdot) h^2 \mathrm{d} \mu \leqslant C_3 \int\left| \nabla h \right|^2 \mathrm{d}\mu + C_4\int h^2\mathrm{d}\mu,  \label{eq5-1}
  \end{eqnarray}
and then implies $\textrm{W}_2\textrm{I}$ by \cite{Liuy}, which partially depends on two facts that (\ref{eq5-1}) implies $\textrm{W}_2\textrm{H}$ and $\textrm{W}_2\textrm{H}$ has a Bobkov-G\"otze's characterization. 

Now there appears another way. For unbounded manifolds,  (\ref{eq5-1}) implies there exists some $r>0$ such that
    \[ \int d^2(x_0, \cdot) h^2 \mathrm{d} \mu \leqslant C_5 \int\left| \nabla h \right|^2 \mathrm{d}\mu + C_6\int_{d(x_0,\cdot) \leqslant r} h^2\mathrm{d}\mu, \]
which leads to PI by \cite{BBCG}. Then using Theorem \ref{Thm2} and (\ref{eq5-1}) and PI yields
   \[ W_2^2(\nu,\mu) \leqslant 2C_1\int \left(d^2(x_0,\cdot)+\mu\left(d^2(x_0,\cdot)\right)\right)(\sqrt{f} -c)^2 \mathrm{d}\mu + C_2 \sigma^2
        \leqslant C_7\mathrm{I}_\mu(\nu|\mu), \]
where we use the fact that for any $x$ and any bounded $h$ with $\mu(h) = 0$ holds
    \[ Q_1h(x) \leqslant \int h(y) + \frac12d^2(x,y) \mathrm{d}\mu(y) \leqslant d^2(x_0,\cdot)+\mu\left(d^2(x_0,\cdot)\right). \]
Hence we reach $\textrm{W}_2\textrm{I}$.

\subsection*{Acknowledgements}

{\small I am so grateful to the anonymous referee for his/her conscientious reading with many suggestions and comments on the first version of this paper. I also thank Prof. Li-Ming Wu and Prof. Feng-Yu Wang very much for their kindly comments during the conferences held in WHU and BNU respectively. The author is supported by NSFC (no. 11431014, no. 11688101), AMSS research grant (no. Y129161ZZ1), and Key Laboratory of Random Complex Structures and Data, Academy of Mathematics and Systems Science, Chinese Academy of Sciences (No. 2008DP173182).}

%{\small The authors sincerely thank}

% The above example implies that $P^n$ is asymptotic stable on $X$, but the unique ergodic support is $\{0\}$, which has no interior. So Condition $(\mathbf{A}1)$ is not necessary, however, the next example shows that it is sharp.

%\begin{example} \label{exm02}
%Let $\mathbf{\Psi}_n = \{(x,y): x\in [0,1], y=2^{-n}\}$ and $\mathbf{\Psi}_* = \{(x,y): x\in [0,1], y=0\}$. Define $X= \mathbf{\Psi}_* \cup \mathbf{\Psi}_1 \cup \mathbf{\Psi}_2 \cdots$ and the transition probability
%  \[ P((x,y),(2x,y/2)) = 1,  \ \ \ 2x = x+x \;(\mathrm{mod}\; 1). \]
%Let $z=0$, then $P^n$ is eventual continuous at $z$ but not equicontinuous.
%\end{example}

% BibTeX users please use
% \bibliographystyle{}
% \bibliography{}
%
% Non-BibTeX users please use

\end{document}